\documentclass[12pt]{amsart}
\usepackage{color,graphicx,enumerate,wrapfig,amssymb}
\usepackage{pxfonts}
\usepackage{subfigure}
\usepackage{eucal}

\usepackage{hyperref}

\newcommand{\R}{{\mathbb R}}
\newcommand{\Z}{{\mathbb Z}}

\newcommand{\cR}{{\mathcal R}}

\newcommand{\cC}{{\mathcal C}}

\newcommand{\Ex}{{\mathrm{exp}}}

\newcommand{\e}{\varepsilon}

\newcommand{\supp}{\operatorname{supp}}

\newtheorem{theorem}{Theorem}[section]

\newtheorem{claim}[theorem]{Claim}

\newtheorem{cor}[theorem]{Corollary}

\newtheorem{definition}{Definition}[section]
\newtheorem{example}[theorem]{Example}

\newtheorem{lem}[theorem]{Lemma}

\newtheorem{remark}[theorem]{Remark}

\numberwithin{equation}{section}

\begin{document}
\title[Homogenization of Dirichlet problem, with applications]{Applications of Fourier analysis in homogenization of Dirichlet problem II. \\ $L^p$ estimates}

\author{Hayk Aleksanyan}

\address{School of Mathematics, The University of Edinburgh, JCMB The King's Buildings, Mayfield Road, Edinburgh EH9 3JZ, UK}
\email{H.Aleksanyan@sms.ed.ac.uk}
\thanks{H. Aleksanyan thanks G\"{o}ran Gustafsson Foundation for visiting appointment to KTH}

\author[Henrik Shahgholian ]{Henrik Shahgholian}
\address{Department of Mathematics, KTH Royal Institute of Technology,
  100~44  Stockholm, Sweden}
\email{henriksh@kth.se}
\thanks{H. Shahgholian was partially supported by Swedish Research Council}

\author{Per Sj\"{o}lin}
\address{Department of Mathematics, KTH Royal Institute of Technology,
  100~44  Stockholm, Sweden}
\email{persj@math.kth.se}

\subjclass{Primary 35B27; Secondary 42B20}

\begin{abstract}
Let $u_\e$ be a solution to  the system
$$
 \mathrm{div}(A_\e(x)  \nabla u_{\e}(x))=0 \text{ \  in } D, \qquad
u_{\e}(x)=g(x,x/\e) \text{  \ on }\partial D,
$$
where $D \subset \R^d $ ($d \geq 2$), is a smooth uniformly  convex domain, and $g$ is $1$-periodic in its second variable, and both $A_\e$ and $g$ reasonably smooth.

Our results in this paper are two folds. First we prove $L^p$ convergence results for solutions of the above system, for non-oscillating operator, $A_\e(x) =A(x)$, with the following convergence rate for all $1\leq p <\infty$
 $$
 \|u_\e - u_0\|_{L^p(D)} \leq C_p  \begin{cases} \e^{1/2p} ,&\text{  $d=2$}, \\
 (\e  |\ln \e |)^{1/p}, &\text{  $d = 3$ }, \\ \e^{1/p} ,&\text{  $d \geq 4$,}    \end{cases}
 $$
 which we prove is (generically) sharp for $d\geq 4$. Here $u_0$ is the solution to the averaging problem.

 Second,  combining our method with the recent results due to Kenig, Lin and Shen
\cite{KLS1}, we prove (for certain class of operators and when $d\geq 3$ )
$$
|| u_\e - u_0 ||_{L^p(D)} \leq C_p [ \e (\ln(1/ \e))^2 ]^{1/p}.
$$
for both  oscillating operator and boundary data. For this case, we  take $A_\e=A(x/\e)$, where $A$ is $1$-periodic as well.

Some further applications of the method to the homogenization of Neumann problem with
oscillating boundary data are also considered.
\end{abstract}

\keywords{Homogenization, boundary layer, Elliptic systems, Dirichlet problem, Neumann problem, Oscillatory Integrals}

\maketitle

\section{Introduction and main result}

In this paper we continue our study, initiated in \cite{ASS}, of asymptotic behavior of
solutions to elliptic systems in divergence form
\begin{equation}\label{problem-fixed-oper}
 -\mathrm{div}(A(x) \nabla u (x)) =0 , \qquad x\in D,
\end{equation}
set in a bounded domain $D\subset \R^d$ ($d\geq 2$) and with oscillating Dirichlet data
\begin{equation}\label{prob-Dir-data-osc}
 u(x)=g \left(x, \frac{x}{\e} \right), \qquad x \in \partial D.
\end{equation}

\noindent As usual $\e>0$ is a small parameter,
$A(x)=(A_{ij}^{ \alpha \beta }(x))$ is $\R^{N^2 \times d^2}$-valued
function defined on $\R^d$, where $1\leq \alpha, \beta \leq d$, $1\leq i,j\leq N$, and $g(x,y)$ is $\R^N$-valued function
defined on $\partial D \times \R^d$.
Using the summation convention of repeated indices the operator in (\ref{problem-fixed-oper}) is defined as\footnote[1]{If not
stated otherwise, througout the text we will use this convention for repeated indices.}
\begin{equation}\label{def-operator-fixed}
-(\mathcal{L}u)_i:=\left(\mathrm{div}(A(\cdot) \nabla u ) \right)_i (x) =  \frac{\partial}{\partial x^{\alpha}} \left[ A^{\alpha \beta }_{ij } ( \cdot )
\frac{\partial u_j}{\partial x^{\beta}}   \right] (x),
\end{equation}
where $u=(u_1,...,u_N)$ and $1\leq i \leq N$. In a similar way we define a family of operators with
rapidly oscillating coefficients, namely for each $\e>0$ we set
\begin{equation}\label{def-operator-osc}
-(\mathcal{L}_\e u)_i:=\left[ \mathrm{div} \left(A \left(\frac{\cdot}{\e} \right) \nabla u \right) \right]_i (x) =
 \frac{\partial}{\partial x^{\alpha}} \left[ A^{\alpha \beta }_{ij } \left( \frac{\cdot }{\e} \right)
\frac{\partial u_j}{\partial x^{\beta}}   \right] (x).
\end{equation}

For the family of operators $\{ \mathcal{L}_\e \}_{\e>0}$
we set $\mathcal{L}_0$ to be the homogenized (effective) operator in a usual sense of the theory of homogenization (see \cite{BLP}).

\subsection{Assumptions}\label{sec-Assump} We will study  problem (\ref{problem-fixed-oper})-(\ref{prob-Dir-data-osc}) under the following hypotheses.
\begin{itemize}
\item[i]  (Periodicity) The boundary vector-valued function $g$ is 1-periodic in its second variable, i.e.
$$
g(x,y+h)=g(x,y), \qquad  \forall x \in \partial D, \ \forall y\in \R^d, \ \forall h\in \Z^d.
$$

\noindent When dealing with operator $\mathcal{L}_\e$ defined in (\ref{def-operator-osc}) we assume that the matrix $A$ is 1-periodic, i.e.
$$
A(x+h)=A(x), \qquad \forall x \in \R^d, \ \forall h\in \Z^d.
$$
\item[ii] (Ellipticity) There exists a constant $c>0$ such that
$$
  c \xi_{\alpha}^i \xi_{\alpha}^i \leq A_{ij}^{\alpha \beta}(x) \xi_{\alpha}^i \xi_{\beta}^j \leq c^{-1} \xi_{\alpha}^i
\xi_{\alpha}^i, \qquad \forall x\in \R^d, \ \forall \xi \in \R^{d\times N}.
$$
\item[iii] (Convexity) We assume that $\partial D$ is a uniformly convex hypersurface, that is all its principal curvatures are bounded away from 0.
\item[iv] (Smoothness) We suppose that the boundary value $g$ in both variables,
 the all elements of $A$, and domain $D$ are sufficiently smooth\footnote[2]{Here we do not aim
to obtain the optimal smoothness, but rather focus on the method itself.}.

\end{itemize}

For each $\e>0$ let $u_\e$ be the solution to Dirichlet problem
(\ref{problem-fixed-oper})-(\ref{prob-Dir-data-osc}), and
let $u_0$ be the solution to the system (\ref{problem-fixed-oper}) with
Dirichlet data
\begin{equation}\label{avg-Dir-data}
u_0(x)=\overline{g}(x), \qquad x \in \partial D,
\end{equation}
where $\overline{g}(x)=\int_{ \mathbb{T}^d} g(x,y) dy$, and $\mathbb{T}^d$ is the
unit torus of $\R^d$. In \cite{ASS} the current authors proved that for each $\kappa>d-1$ there exists a constant $C_\kappa$
so that
\begin{equation}\label{pointwise-est}
 |u_\e(x)- u_0(x)| \leq C_\kappa \frac{ \e^{(d-1)/2} }{d^{\kappa}(x)}, \qquad \forall x\in D,
\end{equation}
where $d(x)$ is the distance of $x$ from the boundary of $D$.
From our pointwise bound (\ref{pointwise-est}) one could easily obtain $L^p$ convergence
of $u_\e$ to $u_0$ inside the domain $D$ with the rate of convergence $\e^{1/2p}$ for all
$1\leq p < \infty$ and in all dimensions starting from two. Nevertheless, the results in
\cite{ASS} not being optimal, raise naturally the question of finding the
optimal rate for $L^p$ convergence.
Some remarks are in order.

\begin{remark}
The pointwise convergence result in \cite{ASS} is stated for the case when the boundary data $g$ only
depends on its periodic variable.
Generalization to the current setting of two variables, given the smoothness of the
boundary data, is straightforward and follows the similar analysis as in \cite{ASS}.
\end{remark}

\begin{remark}
The existence of effective limit $u_0$ for solutions $u_\e$ follows from \cite{LS}.
However the methods of \cite{LS} do not provide any estimates on the rate of convergence.
\end{remark}

\begin{remark}
In a recent work by D. G\'{e}rard-Varet and N. Masmoudi \cite{GM} the authors consider
the following problem
\begin{equation}\label{problem-form-GM}
\mathcal{L}_\e u_\e(x) =0 , \ x \in D \qquad \text{ and } \qquad u_\e(x) = g \left( x, \frac{x}{\e} \right), \ x \in \partial D,
\end{equation}
where $\mathcal{L}_\e$ is defined as in (\ref{def-operator-osc}). The main result of \cite{GM} states that
under the assumptions (i)-(iv) there exists a fixed boundary data $g^*$ so that if $u_0$ is the solution
to the problem
$$
\mathcal{L}_0 u_0 (x) = 0 , \ x \in D \qquad{ and } \qquad u_0(x)= g^*(x), \ x \in \partial D,
$$
then
\begin{equation}\label{GM-L2-convergence}
|| u_\e - u_0 ||_{L^2(D)} \leq C_\alpha \e^\alpha, \qquad \forall \alpha \in \left(0, \frac{d-1}{3d+5} \right).
\end{equation}
Now observe that for an operator with constant coefficients the setting of \cite{GM}, \cite{ASS}, and the current paper
become identical. Then, by the $L^p$ convergence result of \cite{ASS} we may replace the exponent $\alpha$
in (\ref{GM-L2-convergence}) by 1/4 in all dimensions $d\geq 2$. This gives improvement up to dimensions eight including,
 while for $d\geq 10$ the convergence rate in (\ref{GM-L2-convergence}) is better. Another motivation
for this paper was to investigate the optimal convergence rate for the result in \cite{GM}.
In particular we will show that for some class of operators
one can achieve a better convergence rate than that in (\ref{GM-L2-convergence}) (see Theorem \ref{Thm-L-p-oscillating} below).
\end{remark}

\begin{remark}
  Due to the classical work \cite{AL-systems}, by M. Avellaneda and F.-H. Lin,
the non-oscillating boundary data case of  (\ref{problem-form-GM})
 is well understood.
\end{remark}

In this paper we shall strengthen our results on $L^p$ convergence rate in \cite{ASS}.
Our technique uses Fourier analysis methods and depends heavily on the regularity of the operator, and the boundary data, as well as on the regularity, and uniform convexity of the domain.
Although the method is straightforward and computational analysis, it uses refined and  technical (classical) stationary phase analysis, along with estimates of the Poisson kernel.

In Theorem \ref{Thm-L-p-oscillating} we describe a possible setting when our methods can be combined directly with some of the recent results to deal with the problem of homogenization of elliptic systems with rapidly oscillating coefficients and boundary data considered in \cite{GM}.
In section \ref{sec-Neumann} we show that the method presented here can be applied
to study the homogenization of Neumann problem with fixed operator and oscillating boundary data.

The main results of this paper are the following.

\begin{theorem}\label{Thm-L-p}($L^p$-convergence) Let $u_\e$ be the solution to the problem
(\ref{problem-fixed-oper})-(\ref{prob-Dir-data-osc})
and $u_0$ to that of (\ref{problem-fixed-oper}) and (\ref{avg-Dir-data}) under assumptions (i)-(iv).
Then, for all $1\leq p <\infty$ one has
$$
\|u_\e -u_0 \|_{L^p (D)} \leq C_p \begin{cases} \e^{1/2p} ,&\text{  $d=2$}, \\
(\e  |\ln \e |)^{1/p}, &\text{  $d = 3$ }, \\ \e^{1/p} ,&\text{  $d \geq 4$.}    \end{cases}
$$
\end{theorem}

Next, we consider the question of optimality of the $L^p$ convergence rate provided by Theorem \ref{Thm-L-p}. In particular we
prove that in dimensions greater than 3 the convergence rate obtained in Theorem \ref{Thm-L-p} is sharp. For simplicity we will
consider the case of simple equations rather than systems, and will assume that the boundary value $g$ depends only on its oscillating
variable, that is $g:\mathbb{T}^d \rightarrow \mathbb{C}$, and $\overline{g}=\int\limits_{\mathbb{T}^d} g(y) dy$.

\begin{theorem}\label{Thm-Optimality}(Optimality)
Let $N=1$, and $u_\e$ be the solution to (\ref{problem-fixed-oper})-(\ref{prob-Dir-data-osc}) and $u_0$ to that of (\ref{problem-fixed-oper})
and (\ref{avg-Dir-data}) under assumptions (i)-(iv).
Then for each $1\leq p<\infty$ there exists a constant $C_p$ independent of $\e$, such that
$$
\|u_\e -u_0 \|_{L^p(D)} \geq C_p \e^{1/p} \|g -\overline{g}\|_{L^{\infty}( \mathbb{T}^d )}.
$$
\end{theorem}

Theorems \ref{Thm-L-p}-\ref{Thm-Optimality} imply that the convergence rate of homogenization of the Dirichlet problem
with fixed operator and oscillating boundary data is optimal when $d\geq 4$.

Following \cite{KLS1} we set $P_\gamma^k(x)=x_\gamma(0,...,1,0,...)$, with $1$ in the $k$-th position, where
$1\leq \gamma \leq d$ and $1\leq k \leq N$. We also let $\mathcal{L}_\e^*$ to be the
formal adjoint to $\mathcal{L}_\e$, that is the matrix of coefficients of $\mathcal{L}_\e^*$ is $A_{ji}^{\beta \alpha}$.

\begin{theorem}\label{Thm-L-p-oscillating}(Homogenization of elliptic systems)
Assume $d\geq 3$, and that assumptions (i)-(iv) hold. For each $\e>0$ let $u_\e$ be the solution to the following problem
\begin{equation}\label{problem-form-osc}
\mathcal{L}_\e u_\e (x) = 0 , \ x \in D \qquad{ and } \qquad u_\e(x)= g \left(x, \frac{x}{\e} \right), \ x \in \partial D.
\end{equation}
If $\mathcal{L}_\e^* (P_\gamma^k)=0 $ in $D$ for all $1\leq k \leq N$, $1\leq \gamma \leq d$, and $\e>0$
then there exists a fixed boundary data $g^*$ depending on operator, domain and boundary data $g$ so that
if $u_0$ is the solution to the homogenized problem
$$
\mathcal{L}_0 u_0 (x) = 0 , \ x \in D \qquad{ and } \qquad u_0(x)= g^*(x), \ x \in \partial D,
$$
then
$$
|| u_\e - u_0 ||_{L^p(D)} \leq C_p [ \e (\ln(1/ \e))^2 ]^{1/p}.
$$
\end{theorem}

The restriction on the operator in the last Theorem means that a certain family of vector fields in
$\R^d$ must be divergence free. To see this, observe that from the definition of $\mathcal{L}_\e^*$ and
$P_\gamma^k$ we have
\begin{equation}\label{P-is-sol}
0 = \mathcal{L}_\e^* (P_\gamma^k) =
\frac{\partial}{\partial x^\alpha}\left[ A^{\beta \alpha}_{ji}\left( \frac{\cdot}{\e} \right) \frac{\partial (P_\gamma^k)_j}{\partial x^\beta}  \right](x)=
\frac{1}{\e} \frac{\partial  A_{ki}^{\gamma \alpha} }{\partial x^\alpha} \left( \frac{x}{\e} \right), \ \ \forall x\in D, \ \forall \e>0.
\end{equation}
If $\e>0$ is small enough, the domain $(1/ \e ) D$ will contain a lattice cube, hence in view of periodicity of $A$
the condition (\ref{P-is-sol}) is equivalent to
$$
\frac{\partial  A_{ki}^{\gamma \alpha} }{\partial x^\alpha} (x)=0, \qquad \forall x\in \R^d.
$$
Now set $v_{k,i}^{\gamma}(x)=(A_{ki}^{\gamma 1 },... , A_{ki}^{\gamma d})(x) $, where $x \in \R^d$, and
$1 \leq  k,i \leq N$, $1 \leq \gamma \leq d$. We obtain that condition (\ref{P-is-sol}) of Theorem \ref{Thm-L-p-oscillating}
is equivalent to
$$
\mathrm{div}( v_{k,i}^{\gamma})(x) = 0 , \qquad \forall x\in \R^d, \ 1\leq k,i \leq N, \ 1\leq \gamma \leq d.
$$
Observe that for scalar equations ($N=1$) the last condition simply means that
the rows of the coefficient matrix $A$ considered as vector fields in $\R^d$ must be divergence free.

\begin{remark}
In all results above the boundary of the domain is assumed to be strictly   curved in all
directions. On the other extreme, when the boundary consists of flat pieces, the problem
for scalar equations $(N=1)$ is studied in our paper \cite{ASS2}.
 It should be noted that the methods in \cite{ASS2} does not apply, at least not straightforwardly, to  systems.
\end{remark}

\subsection{Preliminaries and Notation} Throughout the text by $\Gamma$ we denote the boundary of the domain $D$.
For each $x\in \R$ $(=\R^1)$ we set $\Ex(x)=e^{2 \pi i x}$. Next, for $x\in \R^d$ by $B(x,r)$ we denote an open ball in $\R^d$
centered at $x$ and with radius $r>0$. For $x\in \R^d$,  if not stated otherwise $| x |$ denotes its standard norm.
Also by $C$, $C_1$, $C_2$ we denote absolute constants that may vary from formulae to formulae.

Before proceeding to  proofs of  main results, we need the following statements.

\begin{lem}\label{Lem-Poisson-est} Let $P(x,y)$, where $x\in D$ and $y\in \Gamma$, be the Poisson kernel for the
operator $\mathcal{L}$ in the domain $D$ under assumptions (i)-(iv). Then for each $\alpha=(\alpha_1,...,\alpha_d) \in \Z^d_{+}$ there exists a
constant $C_{\alpha}$ depending on $\alpha$, domain $D$, and operator $\mathcal{L}$, such that
\begin{equation}\label{Poisson-deriv-est}
|D^{\alpha}_y P(x,y)| \leq C_{\alpha} \frac{1}{|x-y|^{d-1+|\alpha|}}, \qquad x\in D, y \in \Gamma,
\end{equation}

\begin{equation}\label{Poisson-dist-est}
| P(x,y)| \leq C_0 \frac{d(x)}{|x-y|^{d}}, \qquad x\in D, y \in \Gamma,
\end{equation}
where $|\alpha|=|\alpha_1|+...+|\alpha_d|$, and $d(x)$ is the distance of the point $x$ from the boundary $\Gamma$.
\end{lem}
Estimates in (\ref{Poisson-deriv-est}) are proved in \cite{ASS}, Lemma 2.1 (see also \cite{DM}).
For the second estimate (with distance) see \cite{AL-systems}, Theorem 3.

Using (\ref{Poisson-dist-est}) we can establish uniform bounds with respect to $x\in D$ on the surface integral of $|P(x,y)|$,
which we will use later on.

\begin{claim}
Let $P(x,y)$ be as above. Then, there exists a constant $C$ so that
\begin{equation}\label{P-int-is-unif-bdd}
\int\limits_\Gamma |P(x,y)| d \sigma(y) \leq C , \qquad \forall x\in D.
\end{equation}

\end{claim}
\begin{proof}

Fix $x\in D$. Without loss of generality we will assume that $d(x)=|x|$, and the tangent plane
to $\Gamma$ at $0$ is $\{x\in \R^d: \ x_d=0 \}$, since otherwise we may bring $x$ and $\Gamma$ to
these positions by translation and rotation of the coordinate system.
Since $D$ is convex and $d(x)=|x|$ it is clear that $x$ is orthogonal to the tangent plane of $\Gamma$ at 0.
Next, in view of the smoothness of the domain there exists a smooth function $\varphi:\R^{d-1} \to \R$ so that
for some $0<\delta<1$ small, which can be chosen independently of $x$, we have
$$
\Gamma \cap B(0,\delta) = \{(y', \varphi(y')): \ |y'|\leq 10 \delta \} \cap B(0,\delta),
$$
where $y'=(y_1,...,y_{d-1})$. Also, it is clear that $\varphi(0)=\nabla \varphi(0)=0$, from which we get
that $|\varphi(y')| \leq C |y'|^2 $, where $|y'|\leq \delta$ .

It follows from (\ref{Poisson-dist-est}) that to get (\ref{P-int-is-unif-bdd}) it is enough to
show that
$$
\int\limits_{ \Gamma \cap B(0,\delta) } \frac{d\sigma(y)}{|x-y|^d}  \leq C \frac{1}{| x |},
$$
where the constant $C$ is independent of $x$. Now, making a change of variables in the last integral we get
\begin{equation}\label{change-of-var}
\int\limits_{ \Gamma \cap B(0,\delta) } \frac{d\sigma(y)}{|x-y|^d}  \leq
C \int\limits_{ |y'|\leq \delta } \frac{dy'}{| x- (y', \varphi(y') ) |^d} .
\end{equation}

From orthogonality of $x$ to $\{x\in \R^d: \ x_d=0 \}$ and the mentioned properties of $\varphi$ we have
$$
| x- (y', \varphi(y')) |^2 = |x'|^2 + |y'|^2 + x_d^2 -2x_d \varphi(y') + \varphi^2(y') \geq | x |^2 + \frac{1}{2} | y'|,
$$
if $|x|$ and $\delta>0$ are sufficiently small. Using the last inequality from (\ref{change-of-var}),
and integrating in the spherical coordinates we get
\begin{multline*}
\int\limits_{ |y'|\leq \delta } \frac{dy'}{| x- (y', \varphi(y') ) |^d}  \leq C
\int\limits_{ |y'|\leq \delta } \frac{d y'}{ ( | x |^2 + | y' |^2 )^{d/2} } \leq C
\int\limits_0^\delta \frac{t^{d-2}}{ ( | x |^2 + t^2 )^{d/2} } dt \leq \\
C \int\limits_0^\delta \frac{dt}{| x |^2 + t^2 } = C \frac{1}{| x |} \arctan \frac{\delta}{| x |}.
\end{multline*}
Since $d(x)=| x |$ the last expression completes the proof.
\end{proof}

\begin{lem}\label{Lem-coeff} (\cite{ASS}, Lemma 2.3)
If $f\in C^k(\mathbb{T}^d)$ and $\tau \in \R$, then
$$
\sum_{\substack{m \in \Z^d \\ m\neq 0}} \frac{1}{| m |^{\tau}} |c_m(f)| \leq C_{k+\tau} \left( \sum\limits_{ \alpha \in \Z^d_+, \
|\alpha|=k} \| D^{\alpha} f \|_2^2   \right)^{1/2},
$$
provided $k+\tau>d/2$, where $c_m(f)$ is the $m$-th Fourier coefficient of $f$,
and $|\alpha|=\alpha_1+...+\alpha_d$.
\end{lem}

\noindent An immediate consequence of Lemma \ref{Lem-coeff} is the following result.
\begin{lem}\label{Lem-finite-sum-of-coeff}
Let $\tau\in \R$, $\Omega$ be a compact subset of $\R^d$, and a function $f(x,y):\Omega\times \mathbb{T}^d \rightarrow \mathbb{C} $
be periodic in its second variable. Suppose that for all $\alpha\in \Z^d_{+}$ with $|\alpha| \leq k$,  $D^{\alpha}_y f(x,y) $ exists and is
continuous on $\Omega \times \mathbb{T}^d$. Then
$$
\sum_{\substack{m \in \Z^d \\ m\neq 0}} \frac{| c_m(f;x) | }{ | m |^{\tau} } \leq C_{k+\tau, f} , \qquad \forall  x\in \Omega,
$$
provided $k+\tau>d/2$, where $c_m(f;x)$ is the $m$-th Fourier coefficient of $f(x,\cdot)$.
\end{lem}

\section{Proof of $L^p$-convergence result}

\noindent \textbf{Proof of Theorem \ref{Thm-L-p}.} We divide the proof into some steps.

\noindent  \textbf{Step 1. Reduction to local graphs.} Let $z\in \Gamma$ and $r>0$ be small. Then there exists an orthogonal
transformation $\cR$ such that
\begin{equation}\label{def-orth-transform}
(\cR(\Gamma-z)) \cap B(0,r) =\{ (y', \psi(y')): \ |y'|\leq 10 r  \} \cap B(0,r),
\end{equation}
where $y'=(y_1,...,y_{d-1})$, $\psi(0)=\nabla \psi(0)=0$ and $\frac{\partial^2 \psi }{\partial y_j^2}(0)=a_j$, $j=1,2 ,..., d-1$,
 with $a_j>0$, and $\frac{\partial^2 \psi}{\partial y_i \partial y_j}(0)=0$, for $i\neq j$. Also $|D^{\alpha } \psi  | \leq C_{\alpha}$, and $0<c\leq a_1\leq a_2\leq...\leq a_{d-1} \leq C$. We also have
\begin{equation}\label{Lip-est-of-grad}
K_1 |y'| \leq |\nabla \psi(y') | \leq K_2 |y'|,
\end{equation}
where $K_1$ and $K_2$ do not depend on $z$. Now choose $\delta>0$ so small that
\begin{itemize}
\item[(a)] $\delta< \frac{r}{1000}$ and $K_1 \delta <1$,
\item[(b)] $(\ref{Lip-est-of-grad})$ holds for $|y'| \leq \frac{K_1}{4K_2} \delta$,
\item[(c)] $\left| \frac{\partial^2 \psi}{\partial y_i  \partial y_j} \right|  \leq \frac{a_1}{1000d}$ for $i\neq j$ and
$\left| \frac{\partial^2 \psi }{\partial y_j^2} -a_j  \right| \leq \frac{a_j}{100}$ for $|y'|\leq 100 \delta$ when $j=1,2,...,d-1$,
\item[(d)] $|n'|\leq K_1 \delta $ implies that there exists a unique $|y'| \leq \delta$ so that $\nabla \psi(y') =n'$.
\end{itemize}

We remark that $\delta$ is a constant that does not depend on $z$. We have $\Gamma\subset \bigcup\limits_{z\in \Gamma} B(z,\frac
12 L \delta)$ where $L=\frac{K_1}{4 K_2 d}$, hence $\Gamma\subset \bigcup\limits_{k=1}^M B(z^k,\frac 12 L \delta)$ for some
$z^1,...,z^M \in \Gamma$. We take a partition of unity $\sum\limits_{k=1}^M \varphi_k =1$ on $\Gamma$, where
$\mathrm{supp} (\varphi_k) \subset B(z^k, L \delta) $, and $\varphi_k \in C^{\infty}$. Set $B_k=B(z^k,L \delta)$. Recall that
$\overline{g}(x)$ is the average of $g$ on the unit torus with respect to its periodic variable, also denote $g_{\e}(x):=g(x,x / \e)$.

We have
$$
u_{\e}(x)-u_0(x)=\sum\limits_{k=1}^M \int\limits_{\Gamma} P(x,y) [ g_{\e}(y) - \overline{g}(y) ] \varphi_k(y) d\sigma(y) := \sum\limits_{k=1}^M I_k,
$$
where
$$
I_k :=\int\limits_{\Gamma} P(x,y) [ g_{\e}(y) - \overline{g}(y) ] \varphi_k(y) d\sigma(y).
$$

\noindent \textbf{Step 2. Reduction to volume integrals.} Set $z=y-z^k$, then
$$
I_k=\int\limits_{(\Gamma \cap B_k)-z^k} P(x,z^k+z)  [ g_{\e} -\overline{g}](z^k+z) \varphi_k(z^k+z) d\sigma(z).
$$
We have $(\Gamma \cap B_k)-z^k =(\Gamma-z^k) \cap B(0,L \delta)$. By setting $y=\cR z$ we obtain

$$
I_k=\int\limits_{ \cR (\Gamma-z^k) \cap B(0, L\delta) } P(x, z^k+\cR^{-1}y) [ g_{\e}-\overline{g} ](z^k + \cR^{-1} y ) \varphi_k(z^k+\cR^{-1}y) d \sigma(y).
$$
By $(\ref{def-orth-transform})$ and (a) we may assume that
$$
\cR (\Gamma-z^k) \cap B(0, L\delta) =\{ (y', \psi(y')): \ |y'|<100 \delta \} \cap B(0,L \delta),
$$
and hence
$$
I_k=\int\limits_{|y'|<L \delta} P(x, z^k +\cR^{-1}(y',  \psi(y') )  ) [ g_{\e} - \overline{g} ](z^k + \cR^{-1}(y',\psi(y')) ) \cdot
$$
$$
\varphi_k(z^k  +\cR^{-1} (y',  \psi(y') )  ) ( 1+|\nabla \psi(y') |^2  )^{1/2} dy'.
$$

\noindent \textbf{Step 3. Reduction to oscillatory integrals.}  Since $g$ is one periodic in its second variable and sufficiently smooth, we have
$$
g(x,y)=\sum\limits_{m\in \Z^d} c_m(x) \Ex(  m\cdot y ) ,
$$
and hence
$$
g_{\e}(x)=\sum\limits_{m\in \Z^d} c_m(x) \Ex \left( \frac{m}{\e}\cdot x \right),
$$
where $c_m:\Gamma \rightarrow  \mathbb{C}^N$ for each $m\in \Z^d$. Using this and orthogonality of $\cR$ we have
\begin{multline}\label{exp of g}
[g_{\e} -\overline{g}](z^k+\cR^{-1}(y', \psi (y') ) ) = \\
\sum\limits_{m\neq 0} c_m (z^k+\cR^{-1}(y', \psi (y') ) ) \Ex \left( \frac{m}{\e}    \cdot z^k \right)
\Ex \left[ \frac {1}{\e} <\cR m, (y',\psi(y'))>   \right],
\end{multline}
where $<\cdot,\cdot>$ denotes the usual scalar product. By setting $n:=\cR m$ and $n=|n|(n', n_d)$
with $|(n',  n_d)|=1$ from $(\ref{exp of g})$ we obtain
\begin{multline*}
[g_{\e} -\overline{g}](z^k+\cR^{-1}(y', \psi (y') ))=  \\
\sum\limits_{m\neq 0} c_m ( z^k+\cR^{-1}(y', \psi (y') ) ) \Ex \left( \frac{m}{\e} \cdot z^k  \right) \Ex[  \lambda F(y') ],
\end{multline*}
where
$$
F(y')=n' \cdot y' + n_d \psi(y'),
$$
and $\lambda: =\frac{ |n| }{\e}=\frac{|m|}{\e}$ by orthogonality of $\cR$. Next, by setting
$$ \Phi_k(y')= \varphi_k(z^k+\cR^{-1}(y',\psi(y')) ) (1+ |\nabla \psi (y')|^2 )^{1/2}, $$
and
$$
I_{k,m}=\int\limits_{|y'|<L \delta}  c_m(  z^k + \cR^{-1}(y', \psi(y')) )  P(x,  z^k + \cR^{-1}(y', \psi(y')) )
  \Phi_k(y') \Ex [ \lambda F(y') ] d y',
$$
we obtain
$$
|I_k|\leq \sum\limits_{m\neq 0} |I_{k,m}|.
$$

\noindent \textbf{Step 4. Decay of $I_k$.} We split the study of decay of the integrals $I_{k,m}$ into two cases.

\noindent \textbf{Case 1.} $|n'|\geq K_1 \delta/2$.

We have $\nabla F(y')=n'+n_d \nabla \psi(y')$. Then $|n_j'| \geq K_1 \delta/2d $ for some $1\leq j\leq d-1$, hence by
$(\ref{Lip-est-of-grad})$ on $\supp (\Phi_k)$ we have
\begin{equation}\label{case1-est of phase from below}
\left| \frac{\partial F}{\partial y_j} \right| \geq \frac{K_1 \delta}{2d} -K_2|y'| \geq \frac{K_1 \delta}{2d}- K_2 L
\delta =\frac{K_1 \delta}{2d} - \frac{K_1 \delta}{4d}  \geq \frac{K_1 \delta}{4d}.
\end{equation}
Now integrating by parts in $I_{k,m}$ in the $j$-th coordinate twice, by virtue of $(\ref{case1-est of phase from below})$ and
Lemma \ref{Lem-Poisson-est} for all $x\in D$ we conclude
\begin{equation}\label{case1-est of I_km0}
|I_{k,m}(x)| \leq C \lambda^{-2} \int\limits_{|y'| \leq L \delta} \frac{ \left[ |c_m| + \left| \frac{\partial c_m}{\partial y_j}
\right| +\left| \frac{\partial^2 c_m}{\partial y_j^2} \right| \right]  ( z^k + \cR^{-1}(y', \psi(y') )  )  }{ |x -z^k - \cR^{-1}(y', \psi(y') ) |^{d-1+2} } dy'.
\end{equation}
Recall that $d(x)$ is the distance of $x\in D$ from the boundary of $D$, and set
\begin{equation}\label{DEF-D-eps}
D_{\e}=\{x\in D: \ d(x) \geq  \e \}.
\end{equation}
Now observe that
\begin{equation}\label{log}
\int\limits_{ D_{\e} } \frac{1}{ |x -z^k - \cR^{-1}(y', \psi(y') ) |^{d+1} } dx  \leq C \int\limits_{|w| \geq \e}
\frac{dw}{|w|^{d+1}}\leq
C \int\limits_{\e}^C \frac{r^{d-1}}{r^{d+1}} dr = \frac{C}{ \e}.
\end{equation}
Combining this and $(\ref{case1-est of I_km0})$ we obtain
$$
\int\limits_{D_{\e}}|I_{k,m}(x)| dx \leq C \lambda^{-2} \e^{-1} \int\limits_{|y'|\leq L \delta} \left[ |c_m| + \left|
 \frac{\partial c_m}{\partial y_j} \right| +\left| \frac{\partial^2 c_m}{\partial y_j^2} \right| \right]  ( z^k + \cR^{-1}(y', \psi(y') )  ) dy'.
$$
Now taking into account the smoothness properties of $g$ and applying Lemma \ref{Lem-finite-sum-of-coeff} to $g$ and to its
derivatives to sum up $c_m$ and its derivatives, from the last estimate we obtain
\begin{equation}\label{case1-est of I-km}
\sum\limits_{m\neq 0} \|I_{k,m}\|_{L^1(D_{\e})} \leq C \e.
\end{equation}

\noindent \textbf{Case 2.} $|n'|<K_1\delta /2$.

Since $\delta>0$ is small and $|(n',n_d)|=1$ we have $|n_d|> 1/2$ and hence
$$
\left| \frac{n'}{n_d} \right|< \frac{K_1 \delta}{2 \frac 12 } =K_1 \delta.
$$
By (d) there exists a unique $\widetilde{y'}$ with $|\widetilde{y'}|\leq \delta$ and $\nabla \psi (\widetilde{y'}) = -\frac{n'}{n_d}$.
Clearly $\nabla F(\widetilde{y'})=0$, and using (c) we arrive at
$$
\left| \frac{\partial}{\partial y_j} \left( \frac{\partial F}{\partial y_j}  \right) \right| =
\left| n_d \frac{\partial^2 \psi}{\partial y_j^2} \right|
 \geq \frac 12 \frac {99}{100} a_1.
$$
From the latter it follows that
$$
\left| \frac{\partial F}{\partial y_j} ( \widetilde{y'} + s e_j' )  \right| \geq \frac{99}{200} a_1 |s|, \qquad |s|\leq 4\delta,
$$
where $s$ is scalar and $e_j'$ is the $j$-th unit vector of $\R^{d-1}$. By (c) for $i\neq j$ we have
$$
\left| \frac{\partial }{\partial y_i} \left( \frac{\partial F}{\partial y_j} \right)  \right|=\left| n_d
\frac{\partial^2 \psi}{ \partial y_i \partial y_j } \right| \leq \frac{a_1}{1000 d},
$$
from which we obtain
\begin{equation}\label{case2-est-of-deriv-F}
\left| \frac{\partial F}{\partial z_j}(\widetilde{y'} +z') \right|\geq c |z_j|,
\end{equation}
for $z'\in \cC_j$ and $\widetilde{y'}+z' \in \supp (\Phi_k)$ where
$$
\cC_j=\{ z'\in \R^{d-1}: \ |z_j' |\geq \frac{1}{2 \sqrt{d-1} } |z'| \}, \qquad 1,2,... ,d-1.
$$
Clearly the cones $\cC_j$ cover $\R^{d-1}$. For $j=1,2,...,d-1$ there exists $\omega_j$ supported in $\cC_j$, smooth away
from the origin and homogenous of degree 0 such that
$$
\sum\limits_{j=1}^{d-1} \omega_j(z') =1, \qquad \forall z'\neq 0.
$$
Now fix a nonnegative function $h\in C^{\infty}(\R^{d-1})$ such that $h(y')=0$ for $|y'|\geq 2$ and $h(y')=1$ for $|y'|\leq 1$.
Setting $y'=\widetilde{y'}+z'$ and $z^{\ast}:=z^k+\cR^{-1}( \widetilde{y'}+z', \psi(\widetilde{y'}+z') )$ we obtain
$$
I_{k,m}=\int\limits_{|\widetilde{y'}+z'|<L\delta} c_m(z^{\ast}) P(x, z^{\ast}  )
 \Phi_k(\widetilde{y'}+z') \Ex [ \lambda F(\widetilde{y'}+z' ) ] dz'.
$$

Set
$$
I_{k,m}^1= \int\limits_{|\widetilde{y'}+z'|<L\delta} h(\e^{-1/2} z' ) c_m(z^{\ast}) P(x, z^{\ast} )
 \Phi_k (\widetilde{y'}+z' ) \Ex [ \lambda F(\widetilde{y'}+z' ) ] dz',
$$
and
$$
I_{k,m}^2= \int\limits_{|\widetilde{y'}+z'|<L\delta} (1-h(\e^{-1/2} z' ) ) c_m(z^{\ast}) P(x,z^{\ast} )  \Phi_k (\widetilde{y'}+z' )
\Ex [ \lambda F(\widetilde{y'}+z' ) ] dz' ,
$$
so that $I_{k,m}=I_{k,m}^1+I_{k,m}^2$. It follows from Lemma \ref{Lem-Poisson-est} that
$$
\int\limits_{D_{\e}}|P(x,y)| dx \leq C \int\limits_{D_{\e}} \frac{dx}{|x-y|^{d-1}} \leq C,
$$
uniformly with respect to $y\in \Gamma$ and $\e>0$,
which together with the smoothness condition on $g$ and Lemma \ref{Lem-finite-sum-of-coeff} gives
$$
\sum\limits_{m\neq 0}  \int\limits_{D_{\e}} |I_{k,m}^1(x)| dx \leq C   \sum\limits_{m\neq 0} \int\limits_{|z'|\leq 2 \e^{1/2} }
   |c_m(z^{\ast})| dz' \leq C \e^{(d-1)/2} .
$$
For the second part we have $I_{k,m}^2= \sum\limits_{j=1}^{d-1} I_{k,m}^{2,j} $ where
$$
I_{k,m}^{2,j} = \int\limits \omega_j(z') (1-h(\e^{-1/2} z')) c_m(z^{\ast}) P(x,z^{\ast} )  \Phi_k (\widetilde{y'}+z')
\Ex [\lambda F(\widetilde{y'}+z') ] dz'.
$$
Now integrating by parts with respect to $z_j$ in $I_{k,m}^{2,j}$ twice we obtain
\begin{multline}\label{case2-est1-I-km2}
|I_{k,m}^{2,j}(x)| \leq  C \lambda^{-2}  \cdot \\ \int\limits \left| \frac{\partial }{ \partial z_j } \left\{  \frac{1}{  \frac{\partial F}{ \partial z_j } }
\frac{\partial }{ \partial z_j }  [ \frac{1}{  \frac{\partial F}{ \partial z_j } }
 \omega_j(z') (1-h(\e^{-1/2} z')) c_m(z^{\ast}) P(x, z^{\ast})  \Phi_k ] \right\} \right| dz'
\end{multline}

Observe that since $\omega_j$ is homogeneous of degree 0, for each $j=1,2,...,d-1$ and small $|z'|$ we have
$$
\left| \frac{\partial^k \omega_j }{\partial z_j^k} (z')  \right| \leq C \frac{1}{|z'|^k}, \qquad k=1,2, ...
$$
Using this, $(\ref{case2-est-of-deriv-F})$, $(\ref{case2-est1-I-km2})$, Lemma \ref{Lem-Poisson-est}, and applying
Lemma \ref{Lem-finite-sum-of-coeff} we obtain
$$
\sum\limits_{m\neq 0}| I_{k,m}^{2,j} (x)| \leq C \e^2  \sum\limits_{k=1}^6 A_k(x),
$$
where
$$
A_1(x)=\int\limits_{ \e^{1/2} \leq |z'| \leq C } \frac{1 }{|z'|^4} \frac{1}{ |x-z^{\ast} |^{d-1} } dz',
$$

$$
A_2(x)=\e^{-1/2} \int\limits_{ \e^{1/2} \leq |z'| \leq 2 \e^{1/2}  } \frac{1}{|z'|^3} \frac{1}{ |x- z^{\ast} |^{d-1} }  dz',
$$

$$
A_3(x)=\int\limits_{\e^{1/2}  \leq |z'| \leq C } \frac{1}{|z'|^3} \frac{1}{ |x-z^{\ast} |^d }  dz',
$$

$$
A_4(x)= \e^{-1} \int\limits_{\e^{1/2}  \leq |z'| \leq 2 \e^{1/2} } \frac{1}{|z'|^2} \frac{1}{ |x-z^{\ast} |^{d-1} }  dz',
$$

$$
A_5(x)= \e^{-1/2} \int\limits_{\e^{1/2}  \leq |z'| \leq 2 \e^{1/2} } \frac{1}{|z'|^2} \frac{1}{ |x-z^{\ast} |^d }  dz',
$$
and
$$
A_6(x)=  \int\limits_{\e^{1/2}  \leq |z'| \leq C } \frac{1}{|z'|^2} \frac{1}{ |x-z^{\ast} |^{d+1} }  dz'.
$$

An easy calculation shows that
$$
  \int\limits_{D_{\e}} \sum\limits_{k=1}^6 A_k(x) dx \leq C \begin{cases}     \e^{-3/2}  ,&\text{  $d= 2$, }  \\
\e^{-1} | \ln \e | ,&\text{  $d= 3$, } \\     \e^{-1} ,&\text{  $d\geq 4$},  \end{cases}
$$
where we used Fubini's theorem to change the volume and surface integrations. Using this we obtain
$$
\sum\limits_{m\neq 0}  \int\limits_{D_{\e}} |I_{k,m}^{2,j}(x)| dx \leq  C \begin{cases}  \e^{1/2}  ,&\text{ $d= 2$, } \\
\e |\ln \e|  , &  \text{ $d= 3$, } \\   \e ,&\text{ $d\geq 4$.}   \end{cases}
$$
Combining together the estimates for $I_{k,m}^{1,j}$ and $I_{k,m}^{2,j}$, and using $(\ref{case1-est of I-km})$ we arrive at
\begin{equation}\label{est-I-k on D-eps}
\int\limits_{D_{\e}} |I_k(x)| dx \leq C \begin{cases} \e^{1/2}  ,&\text{ $d= 2$, } \\
\e |\ln \e|  , &  \text{ $d= 3$, } \\   \e ,&\text{ $d\geq 4$.}   \end{cases}
\end{equation}

\noindent \textbf{Step 5. $L^p$ estimates.} By virtue of (\ref{P-int-is-unif-bdd})
we have $\sup\limits_{\e>0} \|u_{\e}-u_0\|_{L^{\infty}(D)} <\infty$ and hence
\begin{equation}\label{est-Dirichlet-small-strip}
\int\limits_{D\setminus D_{\e}}|u_{\e}(x) -u_0(x)|dx \leq C \e,
\end{equation}
which combined with (\ref{est-I-k on D-eps}) gives the claim when $p=1$.

Now for $1<p<\infty$ using the boundedness of $u_{\e}-u_0$ we obtain
$$
\int\limits_{D} |u_{\e} -u_0|^p dx \leq C \int\limits_{D} |u_{\e} -u_0|dx =C \| u_{\e} -u_0\|_{L^1(D)},
$$
hence
\begin{equation}\label{est-L-p norm by L-1}
\| u_{\e} -u_0\|_{L^p(D)} \leq C \| u_{\e} -u_0\|_{L^1(D)}^{1/p}.
\end{equation}

Theorem \ref{Thm-L-p} is proved. $\square$

$\newline$

\section{Optimality: proof of Theorem \ref{Thm-Optimality}.}
Throughout this section instead of systems we will consider equations, so the operator $\mathcal{L}$ is considered only in the case $N=1$.
We begin with a simple lemma.

\begin{lem}(Concentration near the boundary) \label{Lem-boundary values}
Let $u$ be the solution to the Dirichlet problem for the operator $\mathcal{L}$ in the  domain $D$ with boundary data $g:\R^d \rightarrow \mathbb{C} $
which is Lipschitz with constant $Lip(g)$.

Then there exist constants $C_1$, $C_2$ depending on dimension, domain, operator,
but independent of $g$, so that for any $x\in D$, $\xi \in \Gamma$,
and small enough $\delta>0$ one has
$$
|u(x) - g (\xi)| \leq  C_1\delta Lip(g)  + \frac 18 \| g \|_{\infty}  ,
$$
provided $|x-\xi| \leq C_2 \delta $.
\end{lem}

\begin{proof}
By the Poisson representation we have
$$
u(x)=\int\limits_{  \Gamma } P(x,y) g(y) d\sigma(y), \qquad x\in D.
$$

Fix $\xi \in \Gamma$ and $x\in D$. If $|x- \xi| \leq \delta/2$, and $|\xi - y|>\delta$
where $ y \in \Gamma$, then clearly $|x-y|>\delta/2$.
Using this, (\ref{P-int-is-unif-bdd}), the second estimate of Lemma \ref{Lem-Poisson-est},
and the fact that the Poisson kernel has integral equal to
one over the boundary $\Gamma$, we obtain
\begin{multline}
|u(x)-g(\xi)|=| \int\limits_{y \in \Gamma} P(x,y) [ g(y)-g(\xi) ] d\sigma(y) | \leq \\
\int\limits_{|y-\xi|<\delta  } |P(x,y)| |g(y) -g(\xi) | d\sigma(y) +
\int\limits_{|y-\xi|\geq \delta  } |P(x,y)| |g(y) -g(\xi) | d\sigma(y) \leq \\
C \delta Lip(g)  +C \|g\|_{\infty}   d(x) \int\limits_{|y -\xi | \geq \delta } \frac{d \sigma(y)}{|x-y|^d} ,
\end{multline}
where the constant $C$ is determined by the Poisson kernel. The last integral is estimated in a similar
way as we proved (\ref{P-int-is-unif-bdd}), and uniformly
with respect to $x$ we obtain
$$
\int\limits_{|y -\xi | \geq \delta } \frac{d \sigma(y)}{|x-y|^d} \leq C \frac{1}{\delta}.
$$
It is left to take $x$ so that $|x- \xi|<C_2\delta$, where $C_2$ is a sufficiently small constant
independent of $x\in D$, $\xi \in \Gamma$ and $g$,
hence the claim.
\end{proof}

The next Lemma is essentially the Weyl's equidistribution theorem, in our case concerning
equidistribution of the scaled surfaces modulo one.
\begin{definition}
For $x, y \in \R^d$ we say that they are equal modulo one, and write $x  \equiv y \ ( \mathrm{mod} \ 1)$ if $x-y \in \Z^d$.
\end{definition}

If $x\in \R^d$, by $x \ \mathrm{mod} \ 1$ we denote the unique point $y$ in the unit torus of $\R^d$ which is equal to $x$ modulo one.

\begin{lem}(Equidistribution of scaled surfaces) \label{Lem-equidis-surface}
Suppose $\Gamma$ is a uniformly convex smooth hypersurface in $\R^d$ ($d\geq 2$). Then for any Riemann integrable function $g:\mathbb{T}^d \rightarrow \R$ one has
\begin{equation}\label{equdis-integral}
\int\limits_{\mathbb{T}^d} g(x) dx = \lim\limits_{\lambda \rightarrow \infty} \frac{1 }{ \mathcal{H}_{d-1} (\Gamma) }   \int\limits_{\Gamma} g(\lambda y) d\sigma(y),
\end{equation}
where $\mathcal{H}_{d-1}$ denotes $(d-1)$-dimensional Hausdorff measure.
\end{lem}
\begin{proof}
We first prove the Lemma for smooth functions. Suppose $g\in C^{\infty}(\R^d)$ and is one periodic. Then
$$
g(x)=\sum\limits_{m\in \Z^d} c_m \Ex ( m\cdot x ), \qquad x\in \mathbb{T}^d,
$$
which converges absolutely. Plugging this expansion into (\ref{equdis-integral}) we see that it
 is enough to prove that
$$
a_{\lambda}:= \sum\limits_{m\neq 0} c_m \int\limits_{\Gamma} \Ex ( \lambda m\cdot y ) d\sigma(y)
$$
converges to 0, as $\lambda \rightarrow \infty$. Denote by $\widehat{\sigma}(\xi)$ the Fourier transform of
the surface measure $\sigma$.
The following estimate is well-known (see \cite{Stein}, chapter VIII, Theorem 1)
$$
|\widehat{\sigma}(\xi)| \leq C |\xi|^{-(d-1)/2}.
$$
Using this estimate we obtain
$$
|a_{\lambda}| \leq \sum\limits_{m\neq 0} |c_m| |\widehat{\sigma}( \lambda m )| \leq C \lambda^{-(d-1)/2}  \sum\limits_{m\neq 0} \frac{|c_m| }{ \| m \|^{(d-1)/2}}.
$$
The last sum converges due to smoothness of $h$, and thus we get the claim for smooth functions.

Now if $g$ is a characteristic function of some rectangle in the unit torus, then it is easy to see that there exist a sequence of smooth
functions $f_n$ and $F_n$, $n=1,2,...$ so that
\begin{itemize}
\item[1.]$ f_n(x)\leq g(x) \leq F_n(x), \qquad x\in \R^d,$
\item[2.] $\lim\limits_{n \rightarrow \infty} \int\limits_{\mathbb{T}^d} [F_n(x)-f_n(x)] dx=0$,
\end{itemize}
from which it follows that (\ref{equdis-integral}) holds true for characteristic functions of rectangles. Clearly it will hold true also
for their linear combinations, i.e. step-functions. Now observe that when $g$ is Riemann integrable function, then the same pointwise bounds
from above and below hold true by means of step-functions, hence the statement
\end{proof}

Applying Lemma \ref{Lem-equidis-surface} to characteristic functions we obtain the following result.
\begin{cor}\label{cor-equidist}
Let $\Gamma $ be as above, and $A\subset \mathbb{T}^d$  be a ball. Then
$$
\mu(A)=\lim\limits_{\lambda \rightarrow \infty}
\frac{  \mathcal{H}_{d-1}  \{ y\in \Gamma: \ \lambda y \ \mathrm{mod} \ 1 \in A \}  }{ \mathcal{H}_{d-1}(\Gamma) },
$$
where $\mu$ denotes the Lebesgue measure in $\R^d$.
\end{cor}

\noindent Now we are ready to complete the proof of Theorem \ref{Thm-Optimality}.

\noindent \textbf{Proof of Theorem \ref{Thm-Optimality}.}
Without loss of generality we may assume that the boundary data $g$ has mean value 0, and hence $u_0=0$.

By the Poisson representation we have
\begin{equation}\label{optim-u_e}
u_\e(x)=\int\limits_{\Gamma } P(x,y) g(y / \e)  d\sigma(y), \qquad x\in D.
\end{equation}

If $g\equiv 0$, then we are done, otherwise set $E:=\{x\in \mathbb{T}^d: \ |g(x)|> 1/2 \|g\|_{\infty}  \}$. Clearly $E$ is an open set,
and by passing to a subset of positive measure, we may assume that $E$ is a ball.

Due to Corollary \ref{cor-equidist} there exists a constant $c_0>0$ so that for all $\e>0$ small enough one has
$$
\frac{1}{\mathcal{H}_{d-1}(\Gamma)} \mathcal{H}_{d-1} \{ y\in \Gamma: \ |g_{\e}(y)|>\frac 12 \|g\|_{\infty} \} > \frac 12 \mu(E).
$$

Now fix $y\in \Gamma$, so that $|g_{\e}(y)|>1/2 \|g\|_{\infty}$, and apply Lemma \ref{Lem-boundary values} with
$$
\delta=\frac{1}{8C_1} \frac{\| g \|_{\infty} }{Lip(g_\e)} = \frac{1}{8C_1}  \frac{\e  \| g  \|_{\infty}}{ Lip(g)} .
$$
We obtain
\begin{equation}\label{optim-u near the bdry}
|u_{\e}(x) - g_{\e}(y)  | < \|g\|_{\infty}/4 , \  \text{ if } x\in D,  \text{ and } |x-y|\leq C \e ,
\end{equation}
where the constant $C$ is independent of $\e$. Since $|g_{\e}(y)|>\|g\|_{\infty}/2$ on a fixed portion of
the boundary for all small enough
$\e>0$, inequality (\ref{optim-u near the bdry}) implies that on a fixed
portion of the strip $B_{\e}:=\{x\in D:  \ d(x)<C \e  \}$ one has
$|u_{\e}(x)|>\|g\|_{\infty}/4 $, where $x\in B_\e$.

Now for $1\leq p<\infty$ taking the $L^p$ norm of $u_{\e}$ only on that
 strip we obtain $\|u_{\e} \|_{L^p(B)} \geq C  \e^{1/p}  \|g\|_{\infty}$,
which proves the theorem. $\square$

$\newline$
We remark here that Theorem \ref{Thm-Optimality} gives sharp bounds on convergence rate of the
homogenization process in dimensions 4 and higher,
and nearly sharp in dimension 3. For $d=2$, and $p=1$ we give an example
 for which the convergence rate is exactly $\e^{1/2}$.

\noindent \textbf{Example (d=2).}
Let $B$ be the unit disc of $\R^2$, and $g(x_1,x_2)=\Ex (x_2 )$. Note that $g$ is one periodic and has mean value 0 in the unit torus.
Consider the following problem:
$$
\begin{cases}
  \Delta u_{\e}(x)=0, &\text{} x\in B,\\
  u_\e(x)=g(x/\e),  &\text{} x\in \partial B.
\end{cases}
$$

To estimate $u_\e$ on $B$ we proceed using the method of stationary phase (see e.g. see \cite{Stein}, chapter VIII). Let $P(x,y)$, where
$|x|<1$, $|y|=1$ be the Poisson kernel for the Laplace operator in $B$. We will consider $u_\e(x)$ only at the points $|x|<1/2$ where
$P(x,y)$ is a smooth function with bounded derivatives. Observe that the only critical points of the phase function $g$ are north and
south poles of the disc, i.e. $n_+:=(0,1)$ and $n_-:=(0,-1)$. It is also clear that these are non-degenerate critical points. Hence we can
invoke the principle of stationary phase (see \cite{Stein}, chapter VIII, Prop. 6) and obtain
$$
u_\e(x)=C \e^{1/2}  [P(x,n_+) e^{\frac{2\pi i}{\e} }+ P(x,n_-) e^{- \frac{2\pi i}{\e}  } ]+ O(\e^{3/2}),
$$
where $O(\e^{3/2})$ is uniform with respect to $|x|<1/2$. Now to see that the two terms in the parentheses do not cancel, it is enough to
restrict $x$ to $\{x=(x_1,x_2) \in B: \ |x|<1/2, \ 1/4<x_2<1/2 \}$. Considering $u_\e$ on this subset we see that $\|u_\e\|_{L^1(B)} \geq C \e^{1/2}$,
which proves that the convergence rate provided by Theorem \ref{Thm-Optimality} in the case $p=1$ and $d=2$ is sharp.

\section{Proof of Theorem \ref{Thm-L-p-oscillating}}

To prove Theorem \ref{Thm-L-p-oscillating},  we will use a recent result
due to Kenig-Lin-Shen \cite{KLS1} to reduce the setting of rapidly oscillating operators to the fixed operator
with oscillating Dirichlet condition, where our method can be applied. We start with some preliminaries.

For $y\in \Gamma$ set  $n(y)=(n_1(y),...,n_d(y))$ to be the unit outward normal to $\Gamma$ at the point $y$.
Let $\widehat{A}_{ij}^{\alpha \beta}$, $1\leq \alpha , \beta \leq d$, $1\leq i,j\leq N$ be the (constant) coefficient matrix
of the homogenized operators $\mathcal{L}_0$, and set $h(y)=(h_{ij}(y))_{N\times N}$ to be the inverse matrix
of $( \widehat{A}^{\alpha \beta } n_\alpha (y) n_\beta(y) )_{N\times N}$,
where $y\in \Gamma$. It is a classical fact that the operator $\mathcal{L}_0$ is elliptic in a sense of Section \ref{sec-Assump} (ii) (see \cite{BLP})
hence the definition of $h(y)$ is correct. Recall that $P_\gamma^k(x)=x_\gamma(0,...,1,0,...)$, with $1$ in the $k$-th position, where
$1\leq \gamma \leq d$, $1\leq k \leq N$, and $\mathcal{L}_\e^*$ is the
formal adjoint to $\mathcal{L}_\e$, that is the matrix of coefficients of $\mathcal{L}_\e^*$ is $A_{ji}^{\beta \alpha}$.
We introduce the matrix of Dirichlet correctors $\Phi_{\e, \gamma}^{*k}=( \Phi_{\e, \gamma}^{* 1k},..., \Phi_{\e, \gamma}^{* Nk} )$
for the operator $\mathcal{L}_\e^*$ in the domain $D$  defined by
\begin{equation}\label{Dir-corr}
\mathcal{L}_\e^* \Phi_{\e, \gamma}^{*k} (x) =0 , \ x \in D \qquad \text{ and } \qquad \Phi_{\e, \gamma}^{*k}(x) = P_\gamma^k(x) , \ x \in \Gamma.
\end{equation}

For $\e>0$ and $y\in \Gamma$ set
\begin{equation}\label{omega}
\omega_\e^{ij}(y)= h_{ik}(y) \cdot \frac{\partial }{\partial n(y)} \{ \Phi_{\e, \gamma}^{* l k} (y) \} \cdot n_\gamma(y) \cdot n_\alpha(y) n_\beta(y) A_{l j}^{\alpha \beta}(y/ \e).
\end{equation}

Also set $g_\e(x)=g(x,x / \e)$, where $x\in \Gamma$. We are now ready to formulate the result we will use from \cite{KLS1}.
\begin{theorem}\label{Thm-KLS}(Theorem 3.9, \cite{KLS1})
Let $d\geq 3$ and assumptions (i)-(iv) hold. Let also $\mathcal{L}_\e (u_\e) =0$ in $D$ and
$u_\e = g_\e $ on $\Gamma$. Then for any $1\leq p <\infty$ one has
$$
|| u_\e - v_\e ||_{L^p(D)} \leq C \{ \e ( \ln [ \e^{-1} M + 2 ] )^2 \}^{1/p} || g_\e  ||_{L^p(\Gamma)},
$$
where $\mathcal{L}_0 (v_\e)=0$ in $D$ and $v_\e= \omega_\e g_\e $ on $\Gamma$,
with $\omega_\e$ defined by (\ref{omega}), and $M$ is the diameter of $D$.
\end{theorem}
We remark that this theorem is proved under some mild regularity conditions on the operator, domain and boundary data.

\noindent \textbf{Proof of Theorem \ref{Thm-L-p-oscillating}.} Under the condition of the theorem we have that
$\mathcal{L}_\e^* (P_\gamma^k)=0$ in $D$ from which we get that $\Phi_{\e, \gamma }^{* k} \equiv  P_\gamma^k  $
 where $1\leq \gamma \leq d$ and $1\leq k\leq N$. Using this and (\ref{omega}) we get
\begin{equation}\label{omega-2}
\omega_\e^{ij} (y ) = h_{ik}(y) n_\gamma(y) n_\gamma(y) n_\alpha(y) n_\beta(y) A_{k j}^{\alpha \beta}(y / \e) =
h_{ik}(y) n_\alpha(y) n_\beta(y) A_{k j}^{\alpha \beta}(y / \e),
\end{equation}
where the last equality is due to the fact that $n(y)n(y)=|n(y)|^2=1$ for all $y\in \Gamma$.
We now proceed to identification of the homogenized boundary data $g^*(x)$.
Recall that since we are working with the family $\mathcal{L}_\e$,
the coefficient matrix $A$ is now assumed to be 1-periodic. Set $c_m(A_{kj}^{\alpha \beta})$ to
be the $m$-th Fourier coefficient of $A_{kj}^{\alpha \beta}$. For the boundary vector-valued function $g(x,y)$ let
 $g_j$ be its $j$-th component, $1\leq j \leq N$, and set $c_m(g_j; x)$ to be the $m$-th
Fourier coefficient of the function $g_j(x,\cdot)$, where $x\in \Gamma$.

Now observe that by virtue of Theorem \ref{Thm-KLS} to get the homogenization of problem
(\ref{problem-form-osc}) it is enough to homogenize $v_\e$. Using (\ref{omega-2}) and Fourier expansion
of $A$ and $g(x,\cdot)$ for the boundary data of $v_\e$ we get
\begin{multline}\label{v-eps}
v_\e(y)= \omega_\e (y) g_\e(y) = h_{ik}(y) n_\alpha(y) n_\beta(y) A_{k j}^{\alpha \beta}(y / \e) g_j(y,y/\e) = \\
h_{ik}(y) n_\alpha(y) n_\beta(y)  \sum\limits_{m\in \Z^d} c_m ( A_{kj}^{\alpha \beta} ) c_{-m}(g_j; y) + \\
h_{ik}(y) n_\alpha(y) n_\beta(y)  \sum_{\substack{m,n\in \Z^d \\ m+n\neq 0}} c_m ( A_{kj}^{\alpha \beta} ) c_{n}(g_j; y) \Ex \left[ \frac{y}{\e} \cdot( m+n) \right].
\end{multline}
Due to the smoothness conditions on $A$ and $g$ their Fourier series converge absolutely,
hence rearrangements in (\ref{v-eps}) are correct. Set $g_i^*(y)$ to be the first term in the right hand side of  (\ref{v-eps}),
we claim that the homogenized boundary data is $g^*(x)=(g_i^*(x))_{i=1}^N$.
To see this define $u_0 $ as the solution to the following problem
$$
\mathcal{L}_0 u_0 (x) =0 , \ x \in D \qquad \text{ and } \qquad u_0(x) = g^* , \ x \in \Gamma.
$$
By the smoothness of the domain, operator and boundary data, the definition of $v_\e$ and $u_0$, it follows from the proof of Theorem \ref{Thm-L-p} that
$$
||v_\e - u_0||_{L^p(D)} \leq C_p \begin{cases}
(\e  |\ln \e |)^{1/p}, &\text{  $d = 3$ }, \\ \e^{1/p} ,&\text{  $d \geq 4$.}    \end{cases}
$$
This in combination with Theorem \ref{Thm-KLS} finishes the proof of our Theorem \ref{Thm-L-p-oscillating}
with homogenized boundary data $g^*$ defined explicitly in terms of operator, domain and boundary data $g$.   $\square$

\section{The Neumann problem}\label{sec-Neumann}

Throughout this section we will assume that $d\geq 3$ and the operator $\mathcal{L}$ is symmetric, i.e.
for its coefficients one has $A=A^*$ or in the explicit form, $A^{\alpha \beta }_{ij} \equiv A^{\beta \alpha}_{ji}$.

As another application of the proof of convergence result for the Dirichlet problem,
we consider homogenization of the Neumann problem, with
oscillating boundary data. Denote by $N(x,y)$ the matrix of Neumann functions
for operator $\mathcal{L}$ in the domain $D$ (see \cite{KLS3} for the definition).

For the operator $\mathcal{L}$ and for some function $F_{\e}:\R^d  \rightarrow \mathbb{C}^N$ consider
the following problem
\begin{equation}\label{problem-formulation-Neumann}
\begin{cases} \mathcal{L} u_{\e} (x)=F_{\e}(x) &\text{ in $D$}, \\
  \frac{\partial u_\e}{\partial \nu} (x)=g(x,x/\e)
&\text{ on $\Gamma$, } \end{cases}
\end{equation}
where $ \left( \frac{\partial u_{\e}}{\partial \nu} \right)^i (x) =n_{\alpha}(x) A_{ij}^{\alpha \beta}(x)
\frac{\partial u_{\e}^j }{\partial x_{\beta}  } $, $1\leq i \leq N$, denotes the conormal derivative, and $n(x)$ is
the outward unit normal to $\Gamma$ at the point $x$. Here for each $\e>0$ one chooses $F_{\e}$ so that the compatibility
condition $\int\limits_D F_\e(x) dx = \int\limits_{\Gamma} g(y,y / \e) d \sigma(y)$ holds true.
In addition we will also assume that $\sup\limits_{\e>0}\| F_\e \|_{\infty} <\infty$.

\begin{theorem} (Neumann Problem) \label{Thm-Neumann}
Let $d\geq 3$, and assume that conditions (i)-(iv) of Section \ref{sec-Assump} and the symmetry condition $A=A^*$ hold true.
Let $u_\e$ be a solution to the system (\ref{problem-formulation-Neumann})
and $u_0$ be a solution to the same
problem where the boundary value $g$ is replaced by $\overline{g}$, and $F_\e$ is replaced by some smooth function $F_0$ to fulfill the
compatibility condition. Set
$$
v_\e(x)=u_\e(x) -\frac{1}{ | \Gamma | } \int\limits_{\Gamma} u_{\e}(y) d\sigma(y) - \int\limits_{D} N(x,y) F_\e (y) dy,
$$
and let $v_0$ be the term corresponding to the homogenized problem. Then for any $1\leq p<\infty $ one has
$$
\|v_\e -v_0    \|_{L^p(D)} \leq  C_p \begin{cases} \e^{1/p} ,&\text{  $d=3$}, \\
\e^{3/2p}, &\text{  $d=4$}, \\
\e^{2/p} |\ln \e |^{1/p} , &\text{  $d\geq 5$}. \end{cases}
$$
\end{theorem}
The reader may wonder about the behavior for the Neumann problem, versus Dirichlet above.
A better convergence rate in higher dimensions is a consequence of the fact that
Neumann kernel has lower order singularity in comparison to Poisson kernel.

The following is an example of problem \eqref{problem-formulation-Neumann}, for which the convergence rate of its solutions is determined by its boundary data.
\begin{example}\label{Exam-Neumann}
For each $\e>0$ take $F_\e=\frac{1}{|D|} \int\limits_{\Gamma} g (y,y/ \e) d \sigma(y) $, and $F_0=\frac{1}{|D|}  \int\limits_{\Gamma}
\overline{g}(y) d\sigma(y) $. Since $g$ is sufficiently smooth function, and $\Gamma$ is a smooth and uniformly convex hypersurface,
after expanding $g$ into its Fourier series with respect to the periodic variable, and applying the principle of stationary phase
(see \cite{Stein}, chapter VIII, Theorem 1) on each summand we get
$$
| F_{\e} - F_0 | \leq C \e^{(d-1)/2}.
$$
Using this and Lemma \ref{Lem-Neumann} below we obtain
$$
\left| \int\limits_{D} N(x,y) (F_\e -F_0) dy \right| \leq C \e^{(d-1)/2} \int\limits_{D} |N(x,y)| dy \leq C \e^{(d-1)/2},
$$
where $C$ is independent of $x\in D$, and $\e>0$. Combining this last estimate with Theorem \ref{Thm-Neumann},
for each $1\leq p<\infty$ we obtain
$$
\| u_\e -u_0 -\frac{1}{ | \Gamma | } \int\limits_{\Gamma} (u_{\e} -u_0) d\sigma(x) \|_{L^p(D)} \leq C_p \begin{cases}
\e^{1/p} ,&\text{  $d=3$}, \\
\e^{3/2p}, &\text{  $d=4$}, \\
\e^{2/p} |\ln \e |^{1/p} , &\text{  $d\geq 5$}. \end{cases}
$$

The example shows, that we will have the same picture, if we take some smooth and one periodic function $F(x)$, and proceed by
taking $F_\e(x)=F(x/ \e)$, and $F_0=\int\limits_{\mathbb{T}^d} F(x)dx$.
\end{example}

\begin{theorem} (Gradient of Neumann solutions) \label{Thm-Neumann-Grad}
Keeping the same conditions and notation of Theorem \ref{Thm-Neumann}, for each $1\leq p<\infty $,
and any $0<\kappa<1/p$ one has
$$
\|\nabla (v_{\e} -v_0) \|_{L^p (D)} \leq C_{p, \kappa} \e^{\kappa}.
$$
\end{theorem}

For the proofs of Theorems \ref{Thm-Neumann} and \ref{Thm-Neumann-Grad} we need some preliminary estimate.
Recall that $N(x,y)$ denotes the matrix of Neumann functions for operator $\mathcal{L}$ in $D$ defined in \cite{KLS3}.
We have the following lemma.

\begin{lem}\label{Lem-Neumann}
Under the assumptions (i)-(iv) of Section \ref{sec-Assump}, symmetry condition $A=A^*$
and $d\geq 3$ for each $\alpha=(\alpha_1,...,\alpha_d) \in \Z^d_{+}$ there exists a constant
$C_{\alpha}$ such that for all $x\in D$ and $y\in \Gamma$ one has
\begin{equation}\label{Neumann1}
|D^{\alpha}_y N(x,y)| \leq C_{\alpha} \frac{1}{|x-y|^{d+|\alpha|-2}},
\end{equation}
and
\begin{equation}\label{Neumann2}
|D^{\alpha}_y \nabla_x N(x,y)| \leq C_{\alpha} \frac{1}{|x-y|^{d+|\alpha|-1}},
\end{equation}
where $|\alpha|=|\alpha_1|+...+|\alpha_d|$.
\end{lem}

\begin{proof}
The case when $|\alpha| \leq 1$, under weaker conditions on operator and domain was treated in \cite{KLS3}.  The case of $|\alpha|=2$,
or even higher orders can be done
by a scaling argument along with up to boundary uniform regularity for solutions to Neumann problems; see  Lemma 2.1 in  \cite{ASS} for a
similar treatment for the Poisson kernel.
\end{proof}

An easy consequence of this lemma is the following bound on the gradients of $u_\e$.
\begin{lem}\label{Lem-Neumann-sol-bdd}
Let $u_{\e}$ be a solution to the problem (\ref{problem-formulation-Neumann}). Then for each $\kappa>0$ there exists a constant $C_{\kappa}$
independent of $\e$ such that
$$
|\nabla u_{\e}(x) |  \leq C_{\kappa} \frac{1}{ d^{\kappa}(x)}, \qquad \forall x\in D.
$$
\end{lem}

\begin{proof}
The following representation is known (see \cite{KLS1}, Section 4)
\begin{equation}\label{neumann-repres}
u_\e(x) -\frac{1}{|\Gamma|} \int\limits_{\Gamma} u_\e(x) dx  = \int\limits_{D} N(x,y) F_{\e} (y) dy + \int\limits_{\Gamma} N(x,y) g_\e(y) d\sigma(y).
\end{equation}
Using the uniform boundedness of $F_\e$ and $g_\e$ with respect to $\e>0$, from (\ref{neumann-repres}) we obtain
\begin{equation}\label{neumann-grad}
|\nabla u_\e (x)| \leq C \int\limits_{D} |\nabla_x N(x,y)| dy + C\int\limits_{\Gamma} |\nabla_x N(x,y)|   d\sigma(y).
\end{equation}
The volume integral in (\ref{neumann-grad}) is bounded by virtue of estimate (\ref{Neumann2}) in Lemma \ref{Lem-Neumann}.
 For the surface integral, again the estimate (\ref{Neumann2}) of Lemma \ref{Lem-Neumann} gives
$$
\int\limits_{\Gamma} |\nabla_x N(x,y)|   d\sigma(y) \leq C \int\limits_{\Gamma} \frac{ d\sigma(y)}{|x-y|^{d-1}} \leq
\frac{C}{d^{\kappa}(x)} \int\limits_{\Gamma} \frac{ d\sigma(y)}{|x-y|^{d-1-\kappa}}.
$$
The last integral is uniformly bounded with respect to $x$ by some constant depending on $\kappa$, hence we obtain the result.
\end{proof}

\noindent \textbf{Proof of Theorem \ref{Thm-Neumann}.} In view of (\ref{neumann-repres}) we have
\begin{multline*}
v_\e(x)=u_{\e}(x)-\frac{1}{ |\Gamma| } \int\limits_{\Gamma} u_{\e}(y) d\sigma(y) -\int\limits_{D} N(x,y) F_\e(y) dy  =
\\ \int\limits_{\Gamma} N(x,y) g_{\e}(y) d \sigma(y).
\end{multline*}
By this, the proof of the theorem basically follows from the proof of Theorem \ref{Thm-L-p}, by simple modification.
Uniform boundedness of $v_{\e}(x)$ with respect to $x\in D$ and $\e>0$ now follows from the estimate of $|N(x,y)|$ provided by
Lemma \ref{Lem-Neumann}. Next, instead of $(\ref{DEF-D-eps})$ we consider $D_{\e}:=\{x\in D: \ d(x)\geq \e^2  \}$.
It is left to replace $P(x,y)$ by $N(x,y)$, and instead of Lemma $\ref{Lem-Poisson-est}$ use estimates of Lemma \ref{Lem-Neumann} in
the proof of Theorem \ref{Thm-L-p}. Note that a better convergence rate in comparison with the Dirichlet problem is due to lower
singularity of the kernel $N(x,y)$ than that of $P(x,y)$. $\square$

$\newline$

\noindent \textbf{Proof of Theorem \ref{Thm-Neumann-Grad}.} By the second part of Lemma \ref{Lem-Neumann} we see that the gradient
of the Neumann matrix with respect to $x$ variable, which is inside the domain, enjoys almost the same regularity properties as that of
the Poisson kernel provided by Lemma \ref{Lem-Poisson-est}. But this regularity is enough to repeat the steps of the proof of
Theorem \ref{Thm-L-p} up to Step 5, and to obtain the same estimates in the region $D_\e$ which is away from the boundary. To
complete the proof of the theorem for $p=1$ we need to prove the analogue of (\ref{est-Dirichlet-small-strip}). Here we use Lemma
\ref{Lem-Neumann-sol-bdd}. Keeping the same notations as in the proof of Theorem \ref{Thm-L-p} for each small $\tau>0$ we have
\begin{equation}\label{neumann-grad-thm-est1}
\int\limits_{D \setminus D_\e} |\nabla(v_\e-v_0)(x)|dx \leq C_{\tau} \int\limits_{D\setminus D_\e} \frac{dx}{d^{\tau}(x)} \leq C_{\tau} \int\limits_0^{\e} \frac{dr}{r^{\tau}}=C_{\tau} \e^{1-\tau},
\end{equation}
which proves the case $p=1$. Now for $1<p<\infty$, take $r>p$. By the H\"{o}lder's inequality we obtain
\begin{equation}\label{neumann-grad-thm-est-Lp}
\| \nabla(v_\e -v_0)\|_{L^p(D)} \leq \| \nabla (v_\e -v_0)\|_{L^1(D)}^{\alpha_r} \| \nabla( v_\e -v_0) \|_{L^r(D)}^{1-\alpha_r},
\end{equation}
where $1/p=\alpha_r+(1-\alpha_r)/r$, and $\alpha_r \in [0,1]$. From which we  conclude $$\alpha_r=\frac 1p \frac{r-p}{r-1}.$$
From Lemma \ref{Lem-Neumann-sol-bdd} we have $\| \nabla( v_\e -v_0) \|_{L^r(D)}^{1-\alpha_r} \leq C_{\tau}$, where $C_{\tau} $ depends on a
small parameter in the lemma. The bound for $L^p$-norm now follows form the case $p=1$ and the fact that $\lim\limits_{r \rightarrow \infty}
\alpha_r=1/p$. $\square$

\end{document}